\documentclass[a4paper]{amsart}
\usepackage{amscd}
\usepackage{amsmath}
\usepackage{amssymb}
\usepackage{amsthm}
\usepackage{bbm}
\usepackage{stmaryrd}

\usepackage[T1]{fontenc}

\newif\ifpdf
\ifx\pdfoutput\undefined
   \pdffalse        
\else
   \pdfoutput=1     
   \pdftrue
\fi

\ifpdf
   \usepackage[pdftex]{graphicx}
\else
   \usepackage{graphicx}
\fi

\numberwithin{equation}{section} \swapnumbers

\newtheorem{satz}{Satz}[section]

\newtheorem{theorem}[satz]{Theorem}
\newtheorem{proposition}[satz]{Proposition}

\newtheorem{lemma}[satz]{Lemma}

\newtheorem{definition}[satz]{Definition}

\newtheorem{remark}[satz]{Remark}

\newtheorem{example}[satz]{Example}

\begin{document}

\hyphenation{li-te-ra-tu-re Wel-fen-gar-ten}

\title[Some refinements of existence results for SPDEs]{Some refinements of existence results for SPDEs driven by Wiener processes and Poisson random measures}
\author{Stefan Tappe}
\address{Leibniz Universit\"{a}t Hannover, Institut f\"{u}r Mathematische Stochastik, Welfengarten 1, 30167 Hannover, Germany}
\email{tappe@stochastik.uni-hannover.de}
\thanks{The author is grateful to an anonymous referee for valuable comments and suggestions.}
\begin{abstract}
We provide existence and uniqueness of global (and local) mild solutions for a general class of semilinear stochastic partial differential equations driven by Wiener processes and Poisson random measures under local Lipschitz and linear growth (or local boundedness, resp.) conditions. The so-called ``method of the moving frame'' allows us to reduce the SPDE problems to SDE problems.
\end{abstract}
\maketitle

\section{Introduction}

Semilinear stochastic partial differential equations (SPDEs) on Hilbert spaces, being of the type
\begin{align}\label{SPDE-Wiener}
\left\{
\begin{array}{rcl}
dZ_t & = & (A Z_t + \alpha(t,Z_t)) dt + \sigma(t,Z_t)dW_t \medskip
\\ Z_0 & = & z_0,
\end{array}
\right.
\end{align}
have widely been studied in the literature, see e.g. \cite{Da_Prato, Prevot-Roeckner, Zhang, Atma-book}. In equation (\ref{SPDE-Wiener}), $A$ denotes the generator of a strongly continuous semigroup, and $W$ is a trace class Wiener process. In view of applications, this framework has been extended by adding jumps to the SPDE (\ref{SPDE-Wiener}). More precisely, consider a SPDE of the type
\begin{align}\label{SPDE-Poisson}
\left\{
\begin{array}{rcl}
dZ_t & = & (A Z_t + \alpha(t,Z_t)) dt + \sigma(t,Z_t)dW_t
\\ && + \int_E \gamma(t,Z_{t-},x) (\mu(dt,dx) - F(dx)dt) \medskip
\\ Z_0 & = & z_0,
\end{array}
\right.
\end{align}
where $\mu$ denotes a Poisson random measure on some mark space $(E,\mathcal{E})$ with $dt \otimes F(dx)$ being its compensator. SPDEs of this type have been investigated in \cite{Marinelli-Prevot-Roeckner, SPDE}, see also \cite{Kotelenez-82, Kotelenez-84, Knoche, P-Z-book, Ruediger-mild, Prevot}, where SPDEs with jump noises have been studied. 

The goal of the present paper is to extend results and methods for SPDEs of the type (\ref{SPDE-Poisson}) in the following directions:
\begin{itemize}
\item We consider more general SPDEs of the form
\begin{align}\label{SPDE}
\left\{
\begin{array}{rcl}
dZ_t & = & (A Z_t + \alpha(t,Z_t)) dt + \sigma(t,Z_t)dW_t 
\\ && + \int_B \gamma(t,Z_{t-},x) (\mu(dt,dx) - F(dx)dt) 
\\ && + \int_{B^c} \gamma(t,Z_{t-},x) \mu(dt,dx) \medskip
\\ Z_0 & = & z_0,
\end{array}
\right.
\end{align}
where $B \in \mathcal{E}$ is a set with $F(B^c) < \infty$. Then, the integral $\int_B$ represents the small jumps, and $\int_{B^c}$ represents the large jumps of the solution process. Similar SDEs have been considered in finite dimension in \cite[Sec. II.2.c]{Jacod-Shiryaev} and in infinite dimension in \cite{Cao}.

\item We will prove the following results (see Theorem \ref{thm-main}) concerning existence and uniqueness of local and global mild solutions to (\ref{SPDE}):
\begin{enumerate}
\item If $(\alpha,\sigma,\gamma|_B)$ are locally Lipschitz and of linear growth, then existence and uniqueness of global mild solutions to (\ref{SPDE}) holds.

\item If $(\alpha,\sigma,\gamma|_B)$ are locally Lipschitz and locally bounded, then existence and uniqueness of local mild solutions to (\ref{SPDE}) holds.

\item If $(\alpha,\sigma,\gamma|_B)$ are locally Lipschitz, then uniqueness of mild solutions to (\ref{SPDE}) holds.
\end{enumerate}
In particular, the result that local Lipschitz and linear growth conditions ensure existence and uniqueness of global mild solutions does not seem to be well-known for SPDEs, as most of the mentioned references impose global Lipschitz conditions. An exception is the reference \cite{Zhang}, where the author treats Wiener process driven SPDEs of the type (\ref{SPDE-Wiener}), even on 2-smooth Banach spaces, and provides existence and uniqueness under local Lipschitz and linear growth conditions. In \cite{Zhang}, the crucial assumption on the operator $A$ is that is generates an analytic semigroup, while our results hold true for every pseudo-contractive semigroup.

\item We reduce the proofs of these SPDE results to the analysis of SDE problems. This is due to the ``method of the moving frame'', which has been presented in \cite{SPDE}. As a direct consequence, we obtain that any mild solution to (\ref{SPDE}) is c\`{a}dl\`{a}g.
\end{itemize}
As just mentioned, we shall utilize the ``method of the moving frame'' from \cite{SPDE}, which allows us to reduce the SPDE problems to SDE problems. Therefore, we will be concerned with SDEs in Hilbert spaces being of the type 
\begin{align}\label{SDE}
\left\{
\begin{array}{rcl}
d Y_t & = & a(t,Y_t) dt + b(t,Y_t)dW_t + \int_B c(t,Y_{t-},x) (\mu(dt,dx) - F(dx)dt) 
\\ && + \int_{B^c} c(t,Y_{t-},x) \mu(dt,dx) \medskip
\\ Y_0 & = & y_0.
\end{array}
\right.
\end{align}
By using the technique of interlacing solutions at jump times (which, in particular cases has been applied e.g. in \cite[Sec. 6.2]{Applebaum} and \cite[Sec. 9.7]{P-Z-book}), we can reduce the SDE (\ref{SDE}) to SDEs of the form
\begin{align}\label{SDE-without-j}
\left\{
\begin{array}{rcl}
d Y_t & = & a(t,Y_t) dt + b(t,Y_t)dW_t + \int_B c(t,Y_{t-},x) (\mu(dt,dx) - F(dx)dt) \medskip
\\ Y_0 & = & y_0
\end{array}
\right.
\end{align}
without large jumps, and for those SDEs suitable techniques and results are available in the literature. This allows us to derive existence and uniqueness results for the SDE (\ref{SDE}), which are subject to the regularity conditions described above. We point out that the reference \cite{Cao} also studies Hilbert space valued SDEs of the type (\ref{SDE}) and provides an existence and uniqueness result considerably going beyond the classical results which impose global Lipschitz conditions. In Section \ref{sec-successive}, we provide a comparison of our existence and uniqueness result for SDEs of the type (\ref{SDE}) with that from \cite{Cao}.

The remainder of this paper is organized as follows: In Section \ref{sec-prelim} we provide the required preliminaries and notation. In Section \ref{sec-SDE} we prove existence and uniqueness results for (local) strong solutions to SDEs of the form (\ref{SDE}), and in Section \ref{sec-SPDE} we prove existence and uniqueness results for (local) mild solutions to SPDEs of the form (\ref{SPDE}) by using the ``method of the moving frame''.

\section{Preliminaries and notation}\label{sec-prelim}

In this section, we provide the required preliminary results and some basic notation.

Throughout this text, let $(\Omega, \mathcal{F}, \mathbb{F}, \mathbb{P})$ with $\mathbb{F} = (\mathcal{F}_t)_{t \geq 0}$ be a filtered probability space satisfying the usual conditions.

Let $U$ be a separable Hilbert space and let $Q \in L(U)$ be a nuclear, self-adjoint, positive definite linear operator. Then, there exist an orthonormal basis $(e_j)_{j \in \mathbb{N}}$ of $U$ and a sequence $(\lambda_j)_{j \in \mathbb{N}} \subset (0,\infty)$ with $\sum_{j \in \mathbb{N}} \lambda_j < \infty$ such that
\begin{align*}
Qe_j = \lambda_j e_j \quad \text{for all $j \in \mathbb{N}$,}
\end{align*}
namely, the $\lambda_j$ are the eigenvalues of $Q$, and each $e_j$
is an eigenvector corresponding to $\lambda_j$. The space $U_0 := Q^{1/2}(U)$, equipped with the inner product
\begin{align*}
\langle u,v \rangle_{U_0} := \langle Q^{-1/2} u, Q^{-1/2} v \rangle_{U},
\end{align*}
is another separable Hilbert space
and $( \sqrt{\lambda_j} e_j )_{j \in \mathbb{N}}$ is an orthonormal basis.
Let $W$ be an $U$-valued $Q$-Wiener process, see \cite[p. 86, 87]{Da_Prato}. For another separable Hilbert space $H$, we denote by $L_2^0(H) := L_2(U_0,H)$ the space of Hilbert-Schmidt
operators from $U_0$ into $H$, which, endowed with the
Hilbert-Schmidt norm
\begin{align*}
\| \Phi \|_{L_2^0(H)} := \bigg( \sum_{j \in \mathbb{N}} \| \Phi ( \sqrt{\lambda_j} e_j ) \|^2 \bigg)^{1/2},
\quad \Phi \in L_2^0(H)
\end{align*}
itself is a separable Hilbert space.

Let $(E,\mathcal{E})$ be a measurable space which we assume to be a
Blackwell space (see \cite{Dellacherie,Getoor}). We remark
that every Polish space with its Borel $\sigma$-field is a Blackwell
space.
Furthermore, let $\mu$ be a time-homogeneous Poisson random measure on
$\mathbb{R}_+ \times E$, see \cite[Def. II.1.20]{Jacod-Shiryaev}.
Then its compensator is of the form $dt \otimes F(dx)$, where $F$ is
a $\sigma$-finite measure on $(E,\mathcal{E})$.

For the following definitions, let $\tau$ be a finite stopping time.
\begin{itemize}
\item We define the new filtration
$\mathbb{F}^{(\tau)} = (\mathcal{F}_t^{(\tau)})_{t \geq 0}$ by
\begin{align*}
\mathcal{F}_t^{(\tau)} := \mathcal{F}_{\tau + t}, \quad t \geq 0.
\end{align*}

\item We define the new $U$-valued process $W^{(\tau)}$ by
\begin{align*}
W_t^{(\tau)} := W_{\tau + t} - W_{\tau}, \quad t \geq 0.
\end{align*}

\item We define the new random measure ${\mu}^{(\tau)}$ on $\mathbb{R}_+ \times E$ by
\begin{align*}
{\mu}^{(\tau)}(\omega;B) := \mu(\omega;B_{\tau(\omega)}), \quad \omega \in \Omega \text{ and } B \in \mathcal{B}(\mathbb{R}_+) \otimes \mathcal{E},
\end{align*}
where we use the notation
\begin{align*}
B_{\tau} := \{ (t + \tau,x) \in \mathbb{R}_+ \times E : (t,x) \in B
\}.
\end{align*}
\end{itemize}

Then, $W^{(\tau)}$ is a $\mathbb{F}^{(\tau)}$-adapted $Q$-Wiener process and $\mu^{(\tau)}$ is a time-homogeneous Poisson random measure relative to the filtration $\mathbb{F}^{(\tau)}$ with compensator $dt \otimes F(dx)$, cf. \cite[Lemma 4.6]{Positivity}. 

\begin{lemma}\label{lemma-F-stopping-time}
Let $\varrho$ be another stopping time. Then, the mapping $(\varrho-\tau)^+$ is a $\mathbb{F}^{(\tau)}$-stopping time.
\end{lemma}

\begin{proof}
For every $t \in \mathbb{R}_+$ we have
\begin{align*}
\{ (\varrho-\tau)^+ \leq t \} = \{ \varrho - \tau \leq t \} = \{ \varrho \leq \tau + t \} \in \mathcal{F}_{\tau + t} = \mathcal{F}_t^{(\tau)},
\end{align*} 
showing that $(\varrho-\tau)^+$ is a $\mathbb{F}^{(\tau)}$-stopping time.
\end{proof}

Denoting by $\mathcal{P}^{(\tau)}$ the predictable $\sigma$-algebra relative to the filtration $\mathbb{F}^{(\tau)}$, we have the following auxiliary result.

\begin{lemma}\label{lemma-pred-meas}
The following statements are true:
\begin{enumerate}
\item The mapping 
\begin{align*}
\theta_{\tau} : \Omega \times \mathbb{R}_+ \rightarrow \Omega \times \mathbb{R}_+, \quad \theta_{\tau}(\omega,t) := (\omega, \tau(\omega)+t) 
\end{align*}
is $\mathcal{P}^{(\tau)}$--$\mathcal{P}$--measurable.

\item The mapping 
\begin{align*}
\vartheta_{\tau} : \Omega \rightarrow \Omega \times \mathbb{R}_+, \quad \vartheta_{\tau}(\omega) := (\omega, \tau(\omega)) 
\end{align*}
is $\mathcal{F}_{\tau}$--$\mathcal{P}$--measurable.
\end{enumerate}
\end{lemma}

\begin{proof}
According to \cite[Thm. I.2.2]{Jacod-Shiryaev}, the system of sets
\begin{align*}
\{ A \times \{ 0 \} : A \in \mathcal{F}_0 \} \cup \{ [\![ 0, \varrho ]\!] : \varrho \text{ is a stopping time} \}
\end{align*}
is a generating system of the predictable $\sigma$-algebra $\mathcal{P}$. For any set $A \in \mathcal{F}_0$ we have
\begin{align*}
\theta_{\tau}^{-1}(A \times \{ 0 \}) = (A \cap \{ \tau = 0\}) \times \{ 0 \} \in \mathcal{P}^{(\tau)}. 
\end{align*}
Furthermore, for any $\mathbb{F}$-stopping time $\varrho$ we have
\begin{align*}
\theta_{\tau}^{-1} ( [\![ 0,\varrho ]\!] ) &= \theta_{\tau}^{-1}( \{ (\omega,t) \in \Omega \times \mathbb{R}_+ : 0 \leq t \leq \varrho(\omega) \} )
\\ &= \{ (\omega,t) \in \Omega \times \mathbb{R}_+ : 0 \leq \tau(\omega) + t \leq \varrho(\omega) \}
\\ &= \{ (\omega,t) \in \Omega \times \mathbb{R}_+ : 0 \leq t \leq \varrho(\omega) - \tau(\omega) \}
\\ &= [\![ 0,\varrho-\tau ]\!] = [\![ 0,(\varrho-\tau)^+ ]\!] \setminus ( \{ \tau > \varrho \} \times \{ 0 \} ) \in \mathcal{P}^{(\tau)},
\end{align*}
where, in the last step, we have used Lemma \ref{lemma-F-stopping-time}. This proves the first statement.

According to \cite[Thm. I.2.2]{Jacod-Shiryaev}, the system of sets
\begin{align*}
\{ A \times \{ 0 \} : A \in \mathcal{F}_0 \} \cup \{ A \times (s,t] : s < t \text{ and } A \in \mathcal{F}_s \}
\end{align*}
is a generating system of the predictable $\sigma$-algebra $\mathcal{P}$. For any set $A \in \mathcal{F}_0$ we have
\begin{align*}
\vartheta_{\tau}^{-1}(A \times \{ 0 \}) = A \cap \{ \tau = 0\} \in \mathcal{F}_0 \subset \mathcal{F}_{\tau}. 
\end{align*}
Furthermore, for all $s,t \in \mathbb{R}_+$ with $s < t$ and $A \in \mathcal{F}_s$ we have
\begin{align*}
\vartheta_{\tau}^{-1}(A \times (s,t]) = A \cap \{ s < \tau \} \cap \{ \tau \leq t \} \in \mathcal{F}_{\tau},
\end{align*}
establishing the second statement.
\end{proof}

Let us further investigate the Poisson random measure $\mu$. According to \cite[Prop. II.1.14]{Jacod-Shiryaev}, there exist a sequence $(\kappa_n)_{n \in \mathbb{N}}$ of finite stopping times with
$[\![ \kappa_n ]\!] \cap [\![ \kappa_m ]\!] = \emptyset$ for $n \neq m$ and an $E$-valued optional process $\xi$ such that for every optional process $\gamma : \Omega \times \mathbb{R}_+ \times E \rightarrow H$, where $H$ denotes a separable Hilbert space, and all $0 \leq t \leq u$ with
\begin{align*}
\mathbb{P} \bigg( \int_t^u \int_E \| \gamma(s,x) \| \mu(ds,dx) < \infty \bigg) = 1
\end{align*}
we have
\begin{align}\label{integrate-mu}
\int_t^u \int_E \gamma(s,x) \mu(ds,dx) = \sum_{n \in \mathbb{N}} \gamma(\kappa_n,\xi_{\kappa_n}) \mathbbm{1}_{\{ t < \kappa_n \leq u \}}.
\end{align}
Let $B \in \mathcal{E}$ be a set with $F(B^c) < \infty$.
We define the mappings $\varrho_k : \Omega \rightarrow \overline{\mathbb{R}}_+$, $k \in \mathbb{N}_0$ as
\begin{align*}
\varrho_k := \inf \{ t \geq 0 : \mu([0,t] \times B^c) = k \}, \quad k \in \mathbb{N}_0.
\end{align*}

\begin{lemma}\label{lemma-rho}
The following statements are true:
\begin{enumerate}
\item For each $k \in \mathbb{N}$ the mapping $\varrho_k$ is a finite stopping time.

\item We have $\varrho_0 = 0$ and $\mathbb{P}(\varrho_k < \varrho_{k+1}) = 1$ for all $k \in \mathbb{N}_0$.

\item We have $\mathbb{P}(\varrho_k \rightarrow \infty) = 1$.
\end{enumerate}
\end{lemma}

\begin{proof}
This follows from \cite[Lemma A.19]{Manifolds}.
\end{proof}

\section{Existence and uniqueness of strong solutions to Hilbert space valued SDEs}\label{sec-SDE}

In this section, we establish existence and uniqueness of (local) strong solutions to Hilbert space valued SDEs of the type (\ref{SDE}).

Let $\mathcal{H}$ be a separable Hilbert space and let $B \in \mathcal{E}$ be a set with $F(B^c) < \infty$. Furthermore, let $a : \Omega \times \mathbb{R}_+ \times \mathcal{H} \rightarrow \mathcal{H}$ and $b : \Omega \times \mathbb{R}_+ \times \mathcal{H} \rightarrow L_2^0(\mathcal{H})$ be $\mathcal{P} \otimes \mathcal{B}(\mathcal{H})$-measurable mappings, and let $c : \Omega \times \mathbb{R}_+ \times \mathcal{H} \times E \rightarrow \mathcal{H}$ be a $\mathcal{P} \otimes \mathcal{B}(\mathcal{H}) \otimes \mathcal{E}$-measurable mapping.

\begin{definition}
We say that \emph{existence} of (local) strong solutions to (\ref{SDE}) holds, if for each $\mathcal{F}_0$-measurable random variable $y_0 : \Omega \rightarrow \mathcal{H}$ there exists a (local) strong solution to (\ref{SDE}) with initial condition $y_0$ (and some strictly positive lifetime $\tau > 0$). 
\end{definition}

\begin{definition}
We say that \emph{uniqueness} of (local) strong solutions to (\ref{SDE}) holds, if for two (local) strong solutions to (\ref{SDE}) with initial conditions $y_0$ and $y_0'$ (and lifetimes $\tau$ and $\tau'$) we have up to indistinguishability
\begin{align*}
Y \mathbbm{1}_{\{ y_0 = y_0' \}} &= Y' \mathbbm{1}_{\{ y_0 = y_0' \}}
\\ \big( Y^{\tau \wedge \tau'} \mathbbm{1}_{\{ y_0 = y_0' \}} &= (Y')^{\tau \wedge \tau'} \mathbbm{1}_{\{ y_0 = y_0' \}} \big).
\end{align*}
\end{definition}

Note that uniqueness of local strong solutions to (\ref{SDE}) implies uniqueness of strong solutions to (\ref{SDE}). This is seen by setting $\tau := \infty$ and $\tau' := \infty$.

\begin{definition}
We say that the mappings $(a,b,c|_B)$ are \emph{locally Lipschitz}, if $\mathbb{P}$--almost surely
\begin{align*}
\bigg( \int_B \| c(t,y,x) \|^2 F(dx) \bigg)^{1/2} < \infty \quad \text{for all $t \in \mathbb{R}_+$ and all $y \in \mathcal{H}$,}
\end{align*}
and for each $n \in \mathbb{N}$ there is a non-decreasing function $L_n : \mathbb{R}_+ \rightarrow \mathbb{R}_+$ such that $\mathbb{P}$--almost surely
\begin{align}\label{loc-Lipschitz-a}
\| a(t,y_1) - a(t,y_2) \| &\leq L_n(t) \| y_1 - y_2 \|,
\\ \label{loc-Lipschitz-b} \| b(t,y_1) - b(t,y_2) \|_{L_2^0(\mathcal{H})} &\leq L_n(t) \| y_1 - y_2 \|,
\\ \label{loc-Lipschitz-c} \bigg( \int_B \| c(t,y_1,x) - c(t,y_2,x) \|^2 F(dx) \bigg)^{1/2} &\leq L_n(t) \| y_1 - y_2 \|
\end{align}
for all $t \in \mathbb{R}_+$ and all $y_1,y_2 \in \mathcal{H}$ with $\| y_1 \|, \| y_2 \| \leq n$.
\end{definition}

\begin{definition}
We say that the mappings $(a,b,c|_B)$ satisfy the \emph{linear growth condition}, if there exists a non-decreasing function $K : \mathbb{R}_+ \rightarrow \mathbb{R}_+$ such that $\mathbb{P}$--almost surely
\begin{align}\label{lin-growth-a}
\| a(t,y) \| &\leq K(t) ( 1 + \| y \|),
\\ \label{lin-growth-b} \| b(t,y) \|_{L_2^0(H)} &\leq K(t) ( 1 + \| y \|),
\\ \label{lin-growth-c} \bigg( \int_B \| c(t,y,x) \|^2 F(dx) \bigg)^{1/2} &\leq K(t) ( 1 + \| y \|)
\end{align}
for all $t \in \mathbb{R}_+$ and all $y \in \mathcal{H}$.
\end{definition}

\begin{definition}
We say that the mappings $(a,b,c|_B)$ are \emph{locally bounded}, if for each $n \in \mathbb{N}$ there is a non-decreasing function $M_n : \mathbb{R}_+ \rightarrow \mathbb{R}_+$ such that $\mathbb{P}$--almost surely
\begin{align*}
\| a(t,y) \| &\leq M_n(t),
\\ \| b(t,y) \|_{L_2^0(H)} &\leq M_n(t),
\\ \bigg( \int_B \| c(t,y,x) \|^2 F(dx) \bigg)^{1/2} &\leq M_n(t)
\end{align*}
for all $t \in \mathbb{R}_+$ and all $y \in \mathcal{H}$ with $\| y \| \leq n$.
\end{definition}

For a finite stopping time $\tau$ and a set $\Gamma \in \mathcal{F}_{\tau}$ we define the mappings $a^{(\tau,\Gamma)} : \Omega \times \mathbb{R}_+ \times \mathcal{H} \rightarrow \mathcal{H}$, $b^{(\tau,\Gamma)} : \Omega \times \mathbb{R}_+ \times \mathcal{H} \rightarrow L_2^0(\mathcal{H})$ and $c^{(\tau,\Gamma)} : \Omega \times \mathbb{R}_+ \times \mathcal{H} \times E \rightarrow \mathcal{H}$ as
\begin{align}\label{def-a-tau}
a^{(\tau,\Gamma)}(t,y) &:= a(\tau+t,y) \mathbbm{1}_{\Gamma},
\\ \label{def-b-tau} b^{(\tau,\Gamma)}(t,y) &:= b(\tau+t,y) \mathbbm{1}_{\Gamma},
\\ \label{def-c-tau} c^{(\tau,\Gamma)}(t,y,x) &:= c(\tau+t,y,x) \mathbbm{1}_{\Gamma}.
\end{align}
By Lemma \ref{lemma-pred-meas}, the mappings $a^{(\tau,\Gamma)}$ and $b^{(\tau,\Gamma)}$ are $\mathcal{P}^{(\tau)} \otimes \mathcal{B}(\mathcal{H})$-measurable, and $c^{(\tau,\Gamma)}$ is $\mathcal{P}^{(\tau)} \otimes \mathcal{B}(\mathcal{H}) \otimes \mathcal{E}$-measurable.We shall also use the notation
\begin{align}\label{def-abc}
a^{(\tau)} := a^{(\tau,\Omega)}, \quad b^{(\tau)} := b^{(\tau,\Omega)} \quad \text{and} \quad c^{(\tau)} := c^{(\tau,\Omega)}.
\end{align}

\begin{lemma}\label{lemma-Lipschitz-trans}
Suppose that $\tau \mathbbm{1}_{\Gamma}$ is bounded. Then, the following statements are true:
\begin{enumerate}
\item If $(a,b,c|_B)$ are locally Lipschitz, then $(a^{(\tau,\Gamma)},b^{(\tau,\Gamma)},c^{(\tau,\Gamma)}|_B)$ are locally Lipschitz, too.

\item If $(a,b,c|_B)$ satisfy the linear growth condition, then $(a^{(\tau,\Gamma)},b^{(\tau,\Gamma)},c^{(\tau,\Gamma)}|_B)$ satisfy the linear growth condition, too.
\end{enumerate}
\end{lemma}

\begin{proof}
Suppose that $(a,b,c|_B)$ satisfy the linear growth condition. Since $\tau \mathbbm{1}_{\Gamma}$ is bounded, there exists a constant $T \geq 0$ such that $\tau \mathbbm{1}_{\Gamma} \leq T$. The mapping $\tilde{K} := K(\bullet + T) : \mathbb{R}_+ \rightarrow \mathbb{R}_+$ is non-decreasing, and we have $\mathbb{P}$--almost surely
\begin{align*}
\| a^{(\tau,\Gamma)}(t,y) \| = \| a(t + \tau,y) \mathbbm{1}_{\Gamma} \| \leq K(t+\tau) \mathbbm{1}_{\Gamma} (1+\| y \|) \leq \tilde{K}(t) (1+\| y \|)
\end{align*}
for all $t \in \mathbb{R}_+$ and $y \in \mathcal{H}$.
Analogous estimates for $b^{(\tau,\Gamma)}$ and $c^{(\tau,\Gamma)}$ prove that $(a^{(\tau,\Gamma)},b^{(\tau,\Gamma)},c^{(\tau,\Gamma)}|_B)$ satisfy the linear growth condition, too. The remaining statement is proven analogously.
\end{proof}

\begin{lemma}\label{lemma-solution-tau-1}
Let $\tau$ and $\varrho$ be two finite stopping times and let $\Gamma \in \mathcal{F}_{\tau}$ be a set with $\Gamma \subset \{ \tau \leq \varrho \}$. If $Y$ is a $\mathbb{F}$-adapted local strong solution to (\ref{SDE}) with lifetime $\varrho$, then 
\begin{align}\label{def-Y-tau}
Y^{(\tau, \Gamma)} := Y_{\tau + \bullet} \mathbbm{1}_{\Gamma}
\end{align}
is a $\mathbb{F}^{(\tau)}$-adapted local strong solution to (\ref{SDE}) with parameters 
\begin{align}\label{para-pre}
a = a^{(\tau,\Gamma)}, \, b = b^{(\tau,\Gamma)}, \, c = c^{(\tau,\Gamma)}, \, W = W^{(\tau)}, \, \mu = \mu^{(\tau)},
\end{align}
initial condition $Y_{\tau} \mathbbm{1}_{\Gamma}$, and lifetime $(\varrho - \tau)^+$.
\end{lemma}

\begin{proof}
The process $Y^{(\tau,\Gamma)}$ given by (\ref{def-Y-tau}) is $\mathbb{F}^{(\tau)}$-adapted, and we have
\begin{align*}
Y_t^{(\tau,\Gamma)} \mathbbm{1}_{[\![ 0,(\varrho-\tau)^+ ]\!]}(t) &= Y_{\tau + t} \mathbbm{1}_{\Gamma} \mathbbm{1}_{[\![ 0,(\varrho-\tau)^+ ]\!]}(t) = [ Y_{\tau} + (Y_{\tau + t} - Y_{\tau}) ] \mathbbm{1}_{\Gamma} \mathbbm{1}_{[\![ 0,(\varrho-\tau)^+ ]\!]}(t)
\\ &= \bigg[ Y_{\tau} + \int_{\tau}^{\tau + t} a(s,Y_s) ds + \int_{\tau}^{\tau + t} b(s,Y_s) dW_s 
\\ &\qquad + \int_{\tau}^{\tau + t} \int_B c(s,Y_{s-},x) (\mu(ds,dx) - F(dx)ds)
\\ &\qquad + \int_{\tau}^{\tau + t} \int_{B^c} c(s,Y_{s-},x) \mu(ds,dx) \bigg] \mathbbm{1}_{\Gamma} \mathbbm{1}_{[\![ 0,(\varrho-\tau)^+ ]\!]}(t).
\end{align*}
Therefore, we obtain
\begin{align*}
Y_t^{(\tau,\Gamma)}  \mathbbm{1}_{[\![ 0,(\varrho-\tau)^+ ]\!]}(t) &= \bigg[ Y_{\tau} \mathbbm{1}_{\Gamma} + \int_{0}^{t} a(\tau+s,Y_{\tau+s}) \mathbbm{1}_{\Gamma} ds + \int_{0}^{t} b(\tau+s,Y_{\tau+s}) \mathbbm{1}_{\Gamma} dW_s^{(\tau)}
\\ &\qquad + \int_{0}^{t} \int_B c(\tau+s,Y_{(\tau+s)-},x) \mathbbm{1}_{\Gamma} (\mu^{(\tau)}(ds,dx) - F(dx)ds) 
\\ &\qquad + \int_{0}^{t} \int_{B^c} c(\tau+s,Y_{(\tau+s)-},x) \mathbbm{1}_{\Gamma} \mu^{(\tau)}(ds,dx) \bigg] \mathbbm{1}_{[\![ 0,(\varrho-\tau)^+ ]\!]}(t).
\end{align*}
Taking into account the Definitions (\ref{def-a-tau})--(\ref{def-c-tau}) of $a^{(\tau,\Gamma)}$, $b^{(\tau,\Gamma)}$, $c^{(\tau,\Gamma)}$ and the Definition (\ref{def-Y-tau}) of $Y^{(\tau,\Gamma)}$, it follows that 
\begin{align*}
Y_t^{(\tau,\Gamma)} \mathbbm{1}_{[\![ 0,(\varrho-\tau)^+ ]\!]}(t) &= \bigg[ Y_{\tau} \mathbbm{1}_{\Gamma} + \int_{0}^{t} a^{(\tau,\Gamma)}(s,Y_s^{(\tau,\Gamma)}) ds + \int_{0}^{t} b^{(\tau,\Gamma)}(s,Y_s^{(\tau,\Gamma)}) dW_s^{(\tau)}
\\ &\qquad + \int_{0}^{t} \int_B c^{(\tau)}(s,Y_{s-}^{(\tau,\Gamma)},x) (\mu^{(\tau)}(ds,dx) - F(dx)ds) 
\\ &\qquad + \int_{0}^{t} \int_{B^c} c^{(\tau)}(s,Y_{s-}^{(\tau,\Gamma)},x) \mu^{(\tau)}(ds,dx) \bigg] \mathbbm{1}_{[\![ 0,(\varrho-\tau)^+ ]\!]}(t).
\end{align*}
Consequently, $Y^{(\tau,\Gamma)}$ is a local strong solution to (\ref{SDE}) with parameters (\ref{para-pre}),
initial condition $Y_{\tau} \mathbbm{1}_{\Gamma}$, and lifetime $(\varrho - \tau)^+$.
\end{proof}

\begin{lemma}\label{lemma-solution-tau-2}
Let $\tau \leq \varrho$ be two finite stopping times. If $Y^{(0)}$ is a $\mathbb{F}$-adapted local strong solution to (\ref{SDE}) with lifetime $\tau$, and $Y^{(\tau)}$ is a $\mathbb{F}^{(\tau)}$-adapted local strong solution to (\ref{SDE}) with parameters
\begin{align}\label{para-pre-tau}
a = a^{(\tau)}, \, b = b^{(\tau)}, \, c = c^{(\tau)}, \, W = W^{(\tau)}, \, \mu = \mu^{(\tau)},
\end{align}
initial condition $Y_{\tau}^{(0)}$, and lifetime $\varrho - \tau$, then 
\begin{align}\label{def-Y-converse}
Y := Y^{(0)} \mathbbm{1}_{[\![ 0, \tau ]\!]} + Y_{\bullet - \tau}^{(\tau)} \mathbbm{1}_{]\!] \tau,\varrho ]\!]}
\end{align}
is a $\mathbb{F}$-adapted local strong solution to (\ref{SDE}) with lifetime $\varrho$. 
\end{lemma}

\begin{proof}
Let $t \in \mathbb{R}_+$ be arbitrary. Then, the random variable $Y_t^{(0)} \mathbbm{1}_{\{ \tau \geq t \}}$ is $\mathcal{F}_t$-measurable. Let $C \in \mathcal{B}(\mathcal{H})$ be an arbitrary Borel set. We define $D_C \in \mathcal{F}_t$ as
\begin{align*}
D_C :=
\begin{cases}
(\{ \tau < t \} \cap \{ t \leq \varrho \})^c, & \text{if } 0 \in C,
\\ \emptyset, & \text{if } 0 \notin C.
\end{cases}
\end{align*}
According to Lemma \ref{lemma-F-stopping-time}, the mapping
$(t-\tau)^+$ is a $\mathbb{F}^{(\tau)}$-stopping time. Therefore, we get
\begin{align*}
\{ Y_{(t-\tau)^+}^{(\tau)} \in C \} \in \mathcal{F}_{(t-\tau)^+}^{(\tau)} = \mathcal{F}_{\tau + (t-\tau)^+},
\end{align*}
and hence, we obtain
\begin{align*}
&\{ Y_{t-\tau}^{(\tau)} \mathbbm{1}_{\{ \tau < t \leq \varrho \}} \in C \} = \{ Y_{(t-\tau)^+}^{(\tau)} \mathbbm{1}_{\{ \tau < t \leq \varrho \}} \in C \} 
\\ &= ( \{ \tau < t \} \cap \{ t \leq \varrho \} \cap \{ Y_{(t-\tau)^+}^{(\tau)} \in C \} ) \cup D_C
\\ &= ( \{ \varrho \geq t \}  \cap \{ \tau \neq t \} \cap \{ \tau \leq t \} \cap \{ Y_{(t-\tau)^+}^{(\tau)} \in C \} ) \cup D_C
\\ &= ( \{ \varrho \geq t \} \cap \{ \tau \neq t \} \cap \{ \tau + (t-\tau)^+ = t \} \cap \{ Y_{(t-\tau)^+}^{(\tau)} \in C \} ) \cup D_C \in \mathcal{F}_t,
\end{align*}
showing that the process $Y$ defined in (\ref{def-Y-converse}) is $\mathbb{F}$-adapted. Moreover, since $Y^{(\tau)}$ is local strong solution to (\ref{SDE}) with initial condition $Y_{\tau}^{(0)}$ and lifetime $\varrho - \tau$, we have
\begin{align*}
Y_{t-\tau}^{(\tau)} \mathbbm{1}_{]\!] \tau,\varrho ]\!]}(t) &= \bigg[ Y_{\tau}^{(0)} + \int_{0}^{t-\tau} a^{(\tau)}(s,Y_{s}^{(\tau)}) ds + \int_{0}^{t-\tau} b^{(\tau)}(s,Y_{s}^{(\tau)}) dW_s^{(\tau)}
\\ &\qquad + \int_{0}^{t-\tau} \int_B c^{(\tau)}(s,Y_{s-}^{(\tau)},x) (\mu^{(\tau)}(ds,dx) - F(dx)ds) 
\\ &\qquad + \int_{0}^{t-\tau} \int_{B^c} c^{(\tau)}(s,Y_{s-}^{(\tau)},x) \mu^{(\tau)}(ds,dx) \bigg] \mathbbm{1}_{]\!] \tau,\varrho ]\!]}(t).
\end{align*}
By the Definitions (\ref{def-a-tau})--(\ref{def-abc}) of $a^{(\tau)}$, $b^{(\tau)}$, $c^{(\tau)}$, we obtain
\begin{align*}
Y_{t-\tau}^{(\tau)} \mathbbm{1}_{]\!] \tau,\varrho ]\!]}(t) &= \bigg[ Y_{\tau}^{(0)} + \int_{0}^{t-\tau} a(\tau+s,Y_{s}^{(\tau)}) ds + \int_{0}^{t-\tau} b(\tau+s,Y_{s}^{(\tau)}) dW_s^{(\tau)}
\\ &\qquad + \int_{0}^{t-\tau} \int_B c(\tau+s,Y_{s-}^{(\tau)},x) (\mu^{(\tau)}(ds,dx) - F(dx)ds) 
\\ &\qquad + \int_{0}^{t-\tau} \int_{B^c} c(\tau+s,Y_{s-}^{(\tau)},x) \mu^{(\tau)}(ds,dx) \bigg] \mathbbm{1}_{]\!] \tau,\varrho ]\!]}(t).
\end{align*}
Therefore, we get
\begin{align*}
Y_{t-\tau}^{(\tau)} \mathbbm{1}_{]\!] \tau,\varrho ]\!]}(t) &= \bigg[ Y_{\tau}^{(0)} + \int_{\tau}^{t} a(s,Y_{s-\tau}^{(\tau)}) ds + \int_{\tau}^{t} b(s,Y_{s-\tau}^{(\tau)}) dW_s
\\ &\qquad + \int_{\tau}^{t} \int_B c(s,Y_{(s-\tau)-}^{(\tau)},x) (\mu(ds,dx) - F(dx)ds) 
\\ &\qquad + \int_{\tau}^{t} \int_{B^c} c(s,Y_{(s-\tau)-}^{(\tau)},x) \mu(ds,dx) \bigg] \mathbbm{1}_{]\!] \tau,\varrho ]\!]}(t).
\end{align*}
By the Definition (\ref{def-Y-converse}) of $Y$ we obtain
\begin{align*}
Y_{t-\tau}^{(\tau)} \mathbbm{1}_{]\!] \tau,\varrho ]\!]}(t) &= \bigg[ Y_{\tau}^{(0)} + \int_{\tau}^{t} a(s,Y_s) ds + \int_{\tau}^{t} b(s,Y_{s}) dW_s
\\ &\qquad + \int_{\tau}^{t} \int_B c(s,Y_{s-},x) (\mu(ds,dx) - F(dx)ds) 
\\ &\qquad + \int_{\tau}^{t} \int_{B^c} c(s,Y_{s-},x) \mu(ds,dx) \bigg] \mathbbm{1}_{]\!] \tau,\varrho ]\!]}(t).
\end{align*}
Since $Y^{(0)}$ is a local strong solution to (\ref{SDE}) with lifetime $\tau$, we deduce that the process $Y$ given by (\ref{def-Y-converse}) is a local strong solution to (\ref{SDE}) with lifetime $\varrho$.
\end{proof}

Let $k \in \mathbb{N}_0$ be arbitrary. By Lemmas \ref{lemma-F-stopping-time} and \ref{lemma-rho}, the mapping $\varrho_{k+1} - \varrho_k$ is a strictly positive $\mathbb{F}^{(\varrho_k)}$-stopping time. Furthermore, let $\Gamma \in \mathcal{F}_{\varrho_k}$ be arbitrary and let $y_0^{(\varrho_k)} : \Omega \rightarrow \mathcal{H}$ be an arbitrary $\mathcal{F}_0^{(\varrho_k)}$-measurable random variable.

\begin{lemma}\label{lemma-trans-solution-1}
If $Y^{(\varrho_k,\Gamma)}$ is a $\mathbb{F}^{(\varrho_k)}$-adapted local strong solution to (\ref{SDE}) with parameters 
\begin{align}\label{para}
a = a^{(\varrho_k,\Gamma)}, \, b = b^{(\varrho_k,\Gamma)}, \, c = c^{(\varrho_k,\Gamma)}, \, W = W^{(\varrho_k)}, \, \mu = \mu^{(\varrho_k)},
\end{align}
initial condition $y_0^{(\varrho_k)} \mathbbm{1}_{\Gamma}$,
and lifetime $\tau$, then
\begin{align}\label{def-minus}
Y^{(\varrho_k,\Gamma)-} := Y^{(\varrho_k,\Gamma)} - c(\varrho_{k+1},Y_{(\varrho_{k+1} - \varrho_k)-}^{(\varrho_k,\Gamma)},\xi_{\varrho_{k+1}}) \mathbbm{1}_{[\![ \varrho_{k+1} - \varrho_k ]\!]} \mathbbm{1}_{\{ \varrho_{k+1} - \varrho_k \leq \tau \}} \mathbbm{1}_{\Gamma}
\end{align}
is a $\mathbb{F}^{(\varrho_k)}$-adapted local strong solution to (\ref{SDE-without-j}) with parameters (\ref{para}), initial condition $y_0^{(\varrho_k)} \mathbbm{1}_{\Gamma}$, and lifetime $\tau \wedge (\varrho_{k+1} - \varrho_k)$.
\end{lemma}

\begin{proof}
We define $J : \Omega \rightarrow \mathcal{H}$ as
\begin{align*}
J := c(\varrho_{k+1},Y_{(\varrho_{k+1} - \varrho_k)-}^{(\varrho_k,\Gamma)},\xi_{\varrho_{k+1}}) \mathbbm{1}_{\{ \varrho_{k+1} - \varrho_k \leq \tau \}} \mathbbm{1}_{\Gamma}
\end{align*}
and the stochastic process $(J_t)_{t \geq 0}$ as $J_t := J \mathbbm{1}_{[\![ \varrho_{k+1} - \varrho_k ]\!]}(t)$. By Lemma \ref{lemma-pred-meas}, the mapping $J$ is $\mathcal{F}_{\varrho_{k+1}}$-measurable. Let $C \in \mathcal{B}(\mathcal{H})$ be an arbitrary Borel set. We define $D_C \in \mathcal{F}_t$ as
\begin{align*}
D_C :=
\begin{cases}
\{ \varrho_{k+1} - \varrho_k \neq t \}, & \text{if } 0 \in C,
\\ \emptyset, & \text{if } 0 \notin C.
\end{cases}
\end{align*}
Then, for each $t \in \mathbb{R}_+$ we have
\begin{align*}
\{ J_t \in C \} &= \{ J \mathbbm{1}_{[\![ \varrho_{k+1} - \varrho_k ]\!]}(t) \in C \} = ( \{ J \in C \} \cap \{ \varrho_{k+1} - \varrho_k = t \} ) \cup D_C
\\ &= ( \{ J \in C \} \cap \{ \varrho_{k+1} = \varrho_k + t \} ) \cup D_C \in \mathcal{F}_{\varrho_k + t} = \mathcal{F}_t^{(\varrho_k)}.
\end{align*}
Consequently, the process $Y^{(\varrho_k,\Gamma)-}$ defined in (\ref{def-minus}) is $\mathbb{F}^{(\varrho_k)}$-adapted. Furthermore, by the Definition (\ref{def-minus}) we have
\begin{align*}
Y_-^{(\varrho_k,\Gamma)} \mathbbm{1}_{[\![ 0, \tau \wedge (\varrho_{k+1} - \varrho_k) ]\!]} = Y_-^{(\varrho_k,\Gamma)-} \mathbbm{1}_{[\![ 0, \tau \wedge (\varrho_{k+1} - \varrho_k) ]\!]}
\end{align*}
and, by the Definition (\ref{def-c-tau}) of $c^{(\varrho_k,\Gamma)}$ and identity (\ref{integrate-mu}) we obtain
\begin{align*}
&\bigg( \int_0^{t} \int_{B^c} c^{(\varrho_k,\Gamma)}(s,Y_{s-}^{(\varrho_k,\Gamma)},x) \mu^{(\varrho_k)}(ds,dx) \bigg) \mathbbm{1}_{[\![ 0, \tau \wedge (\varrho_{k+1} - \varrho_k) ]\!]}(t)
\\ &= \bigg( \int_0^{t} \int_{B^c} c(\varrho_k + s,Y_{s-}^{(\varrho_k,\Gamma)},x) \mathbbm{1}_{\Gamma} \mu^{(\varrho_k)}(ds,dx) \bigg) \mathbbm{1}_{[\![ 0, \tau \wedge (\varrho_{k+1} - \varrho_k) ]\!]}(t)
\\ &= \bigg( \int_{\varrho_k}^{\varrho_k + t} \int_{B^c} c(s,Y_{(s - \varrho_k)-}^{(\varrho_k,\Gamma)},x) \mathbbm{1}_{\Gamma} \mu(ds,dx) \bigg) \mathbbm{1}_{[\![ 0, \tau \wedge (\varrho_{k+1} - \varrho_k) ]\!]}(t)
\\ &= \bigg( \sum_{n \in \mathbb{N}} c(\kappa_n,Y_{(\kappa_n - \varrho_k) -}^{(\varrho_k,\Gamma)}, \xi_{\kappa_n}) \mathbbm{1}_{\{ \xi_{\kappa_n} \notin B \}} \mathbbm{1}_{\{ \varrho_k < \kappa_n \leq \varrho_k + t \}} \bigg) \mathbbm{1}_{[\![ 0, \tau \wedge (\varrho_{k+1} - \varrho_k) ]\!]}(t) \mathbbm{1}_{\Gamma}
\\ &= c(\varrho_{k+1},Y_{(\varrho_{k+1} - \varrho_k)-}^{(\varrho_k,\Gamma)},\xi_{\varrho_{k+1}}) \mathbbm{1}_{[\![ \varrho_{k+1} - \varrho_k ]\!]}(t) \mathbbm{1}_{\{ \varrho_{k+1} - \varrho_k \leq \tau \}} \mathbbm{1}_{\Gamma},
\end{align*}
showing that $Y^{(\varrho_k,\Gamma)-}$ is a local strong solution to (\ref{SDE-without-j}) with parameters (\ref{para}) and lifetime $\tau \wedge (\varrho_{k+1} - \varrho_k)$.
\end{proof}

\begin{lemma}\label{lemma-trans-solution-2}
If $Y^{(\varrho_k,\Gamma)-}$ is a $\mathbb{F}^{(\varrho_k)}$-adapted local strong solution to (\ref{SDE-without-j}) with parameters (\ref{para}), initial condition $y_0^{(\varrho_k)} \mathbbm{1}_{\Gamma}$, and lifetime $\tau$, then
\begin{align}\label{def-plus}
Y^{(\varrho_k,\Gamma)} := Y^{(\varrho_k,\Gamma)-} + c(\varrho_{k+1},Y_{(\varrho_{k+1} - \varrho_k)-}^{(\varrho_k,\Gamma)-},\xi_{\varrho_{k+1}}) \mathbbm{1}_{[\![ \varrho_{k+1} - \varrho_k ]\!]} \mathbbm{1}_{\{ \varrho_{k+1} - \varrho_k \leq \tau \}}
\mathbbm{1}_{\Gamma}
\end{align}
is a $\mathbb{F}^{(\varrho_k)}$-adapted local strong solution to (\ref{SDE}) with parameters (\ref{para}), initial condition $y_0^{(\varrho_k)} \mathbbm{1}_{\Gamma}$, and lifetime $\tau \wedge (\varrho_{k+1} - \varrho_k)$.
\end{lemma}

\begin{proof}
The proof is analogous to that of Lemma \ref{lemma-trans-solution-1}.
\end{proof}

\subsection{Uniqueness of strong solutions to Hilbert space valued SDEs}

Now, we shall deal with the uniqueness of strong solutions to the SDE (\ref{SDE}).

\begin{proposition}\label{prop-unique}
We suppose that the mappings $(a,b,c|_B)$ are locally Lipschitz. Then, uniqueness of local strong solutions to (\ref{SDE-without-j}) holds.
\end{proposition}

\begin{proof}
We can adopt a standard technique (see, e.g. the proof of Theorem 5.2.5 in \cite{Karatzas-Shreve}), where we apply the It\^{o} isometry and Gronwall's lemma.
\end{proof}

\begin{theorem}\label{thm-SDE-unique}
We suppose that the mappings $(a,b,c|_B)$ are locally Lipschitz. Then, uniqueness of local strong solutions to (\ref{SDE}) holds.
\end{theorem}

\begin{proof}
Let $Y$ and $Y'$ be two local strong solutions to (\ref{SDE-without-j}) with initial conditions $y_0$ and $y_0'$, and lifetimes $\tau$ and $\tau'$.
By induction, we will prove that up to indistinguishability
\begin{align}\label{equal-on-0-k}
Y \mathbbm{1}_{[\![ 0, \tau \wedge \tau' \wedge \varrho_k ]\!]} \mathbbm{1}_{\{ y_0 = y_0' \}} = Y' \mathbbm{1}_{[\![ 0, \tau \wedge \tau' \wedge \varrho_k ]\!]} \mathbbm{1}_{\{ y_0 = y_0' \}} \quad \text{for all $k \in \mathbb{N}_0$.}
\end{align}
The identity (\ref{equal-on-0-k}) holds true for $k = 0$, because by Lemma \ref{lemma-rho} we have $\varrho_0 = 0$. 

For the induction step $k \rightarrow k+1$ we suppose that identity (\ref{equal-on-0-k}) is satisfied. We define the stopping time $\tau_k := \tau \wedge \tau' \wedge \varrho_{k+1}$ and the set $\Gamma_k := \{ \varrho_k \leq \tau_k \} \cap \{ y_0 = y_0' \} \in \mathcal{F}_{\varrho_k}$. By Lemma \ref{lemma-solution-tau-1}, the processes $Y^{(\varrho_k,\Gamma_k)} := Y_{\varrho_k + \bullet} \mathbbm{1}_{\Gamma_k}$ and $Y'^{(\varrho_k,\Gamma_k)} := Y'_{\varrho_k + \bullet} \mathbbm{1}_{\Gamma_k}$ defined according to (\ref{def-Y-tau}) are $\mathbb{F}^{(\varrho_k)}$-adapted local strong solutions to (\ref{SDE}) with parameters (\ref{para-pre}), where $\tau=\varrho_k$ and $\Gamma = \Gamma_k$, initial conditions $Y_{\varrho_k} \mathbbm{1}_{\Gamma_k}$ and $Y'_{\varrho_k} \mathbbm{1}_{\Gamma_k}$, and lifetime $(\tau_k - \varrho_k)^+$. 

Let $n \in \mathbb{N}$ be arbitrary and set $\Gamma_{kn} := \Gamma_k \cap \{ \varrho_k \leq n \} \in \mathcal{F}_{\varrho_k}$. The processes $Y^{(\varrho_k,\Gamma_{kn})} := Y^{(\varrho_k,\Gamma_k)} \mathbbm{1}_{\Gamma_{kn}}$ and $Y'^{(\varrho_k,\Gamma_{kn})} := Y'^{(\varrho_k,\Gamma_k)} \mathbbm{1}_{\Gamma_{kn}}$ are $\mathbb{F}^{(\varrho_k)}$-adapted local strong solutions to (\ref{SDE}) with parameters (\ref{para}), where $\Gamma = \Gamma_{kn}$, initial conditions $Y_{\varrho_k} \mathbbm{1}_{\Gamma_{kn}}$ and $Y_{\varrho_k}' \mathbbm{1}_{\Gamma_{kn}}$, and lifetime $(\tau_k - \varrho_k)^+$.
By Lemma \ref{lemma-trans-solution-1}, the processes $Y^{(\varrho_k,\Gamma_{kn})-}$ and $Y'^{(\varrho_k,\Gamma_{kn})-}$ defined according to (\ref{def-minus})
are $\mathbb{F}^{(\varrho_k)}$-adapted local strong solutions to (\ref{SDE-without-j}) with parameters (\ref{para}), where $\Gamma = \Gamma_{kn}$, initial conditions $Y_{\varrho_k} \mathbbm{1}_{\Gamma_{kn}}$ and $Y'_{\varrho_k} \mathbbm{1}_{\Gamma_{kn}}$, and lifetime $( \tau_k - \varrho_k)^+$. According to Lemma \ref{lemma-Lipschitz-trans}, the mappings $(a^{(\varrho_k,\Gamma_{kn})},b^{(\varrho_k,\Gamma_{kn})},c^{(\varrho_k,\Gamma_{kn})}|_B)$ are locally Lipschitz, too. Therefore, by Proposition \ref{prop-unique} we have up to indistinguishability
\begin{align*}
Y^{(\varrho_k,\Gamma_{kn})-} \mathbbm{1}_{[\![ 0,( \tau_k - \varrho_k)^+ ]\!]} = Y'^{(\varrho_k,\Gamma_{kn})-} \mathbbm{1}_{[\![ 0,( \tau_k - \varrho_k)^+ ]\!]} \quad \text{for all $n \in \mathbb{N}$.}
\end{align*}
By the Definition (\ref{def-minus}), we deduce that up to indistinguishability
\begin{align*}
Y^{(\varrho_k,\Gamma_{kn})} \mathbbm{1}_{[\![ 0,( \tau_k - \varrho_k)^+ ]\!]} = Y'^{(\varrho_k,\Gamma_{kn})} \mathbbm{1}_{[\![ 0,( \tau_k - \varrho_k)^+ ]\!]} \quad \text{for all $n \in \mathbb{N}$,}
\end{align*}
and hence, we have up to indistinguishability
\begin{align*}
Y^{(\varrho_k,\Gamma_k)} \mathbbm{1}_{\{ \varrho_k \leq n \}} \mathbbm{1}_{[\![ 0,( \tau_k - \varrho_k)^+ ]\!]} = Y'^{(\varrho_k,\Gamma_k)} \mathbbm{1}_{\{ \varrho_k \leq n \}} \mathbbm{1}_{[\![ 0,( \tau_k - \varrho_k)^+ ]\!]} \quad \text{for all $n \in \mathbb{N}$.}
\end{align*}
By Lemma \ref{lemma-rho} we have $\mathbb{P}(\varrho_k < \infty) = 1$, and hence, we get up to indistinguishability
\begin{align*}
Y^{(\varrho_k,\Gamma_k)} \mathbbm{1}_{[\![ 0,( \tau_k - \varrho_k)^+ ]\!]} = Y'^{(\varrho_k,\Gamma_k)} \mathbbm{1}_{[\![ 0,( \tau_k - \varrho_k)^+ ]\!]}.
\end{align*}
Therefore, we have up to indistinguishability
\begin{align*}
Y_{\varrho_k + \bullet} \mathbbm{1}_{\{ \varrho_k \leq \tau_k \}} \mathbbm{1}_{[\![ 0,( \tau_k - \varrho_k)^+ ]\!]} \mathbbm{1}_{\{ y_0 = y_0' \}} = Y'_{\varrho_k + \bullet} \mathbbm{1}_{\{ \varrho_k \leq \tau_k \}} \mathbbm{1}_{[\![ 0,( \tau_k - \varrho_k)^+ ]\!]} \mathbbm{1}_{\{ y_0 = y_0' \}}.
\end{align*}
Consequently, we have up to indistinguishability
\begin{align*}
Y \mathbbm{1}_{\{ \varrho_k \leq \tau_k \}} \mathbbm{1}_{[\![ \varrho_k, \tau_k ]\!]} \mathbbm{1}_{\{ y_0 = y_0' \}} = Y' \mathbbm{1}_{\{ \varrho_k \leq \tau_k \}} \mathbbm{1}_{[\![ \varrho_k, \tau_k ]\!]} \mathbbm{1}_{\{ y_0 = y_0' \}}.
\end{align*}
Together with the induction hypothesis, it follows that
\begin{align*}
Y \mathbbm{1}_{[\![ 0,\tau_k ]\!]} \mathbbm{1}_{\{ y_0 = y_0' \}} = Y' \mathbbm{1}_{[\![ 0,\tau_k ]\!]} \mathbbm{1}_{\{ y_0 = y_0' \}},
\end{align*}
which establishes (\ref{equal-on-0-k}). Since by Lemma \ref{lemma-rho} we have $\mathbb{P}(\varrho_k \rightarrow \infty) = 1$, we deduce 
\begin{align*}
Y^{\tau \wedge \tau'} \mathbbm{1}_{\{ y_0 = y_0' \}} = (Y')^{\tau \wedge \tau'} \mathbbm{1}_{\{ y_0 = y_0' \}},
\end{align*}
completing the proof.
\end{proof}

\subsection{Existence of strong solutions to Hilbert space valued SDEs}

Now, we shall deal with the existence of strong solutions to the SDE (\ref{SDE}).

\begin{proposition}\label{prop-SDE-existence}
We suppose that the mappings $(a,b,c|_B)$ are locally Lipschitz and satisfy the linear growth condition. Then, existence of strong solutions to (\ref{SDE-without-j}) holds.
\end{proposition}

\begin{proof}
If the mappings $(a,b,c|_B)$ are Lipschitz continuous, then we have existence and uniqueness of strong solutions to (\ref{SDE-without-j}) for every initial condition $y_0 \in \mathcal{L}^2(\mathcal{F}_0;\mathcal{H})$, see, e.g. \cite[Cor. 10.3]{SPDE}. 

For $(a,b,c|_B)$ being locally Lipschitz and satisfying the linear growth condition, for any initial condition $y_0 \in \mathcal{L}^2(\mathcal{F}_0;\mathcal{H})$ we adopt the technique from the proof of \cite[Thm. 4.11]{Barbara-pathwise}. For $k \in \mathbb{N}$ we define the retraction
\begin{align*}
R_k : \mathcal{H} \rightarrow \mathcal{H}, \quad R_k(y) :=
\begin{cases}
y, & \text{if $\| y \| \leq k$,}
\\ k \frac{y}{\| y \|}, & \text{if $\| y \| > k$,}
\end{cases}
\end{align*}
and the mappings $a_k : \Omega \times \mathbb{R}_+ \times \mathcal{H} \rightarrow \mathcal{H}$, $b_k : \Omega \times \mathbb{R}_+ \times \mathcal{H} \rightarrow L_2^0(\mathcal{H})$ and $c_k : \Omega \times \mathbb{R}_+ \times \mathcal{H} \times E \rightarrow \mathcal{H}$ as
\begin{align*}
a_k := a \circ R_k, \quad b_k := b \circ R_k \quad \text{and} \quad c_k(\bullet,x) := c(\bullet,x) \circ R_k.
\end{align*}
These mappings are Lipschitz continuous, and hence there exists a strong solution $Y^{(k)}$ to the SDE (\ref{SDE-without-j}) with parameters $a = a_k$, $b = b_k$ and $c = c_k$, and initial condition $y_0$. Using the linear growth condition, Gronwall's lemma and Doob's martingale inequality, we can show that $\mathbb{P}(\tau_k \rightarrow \infty) = 1$, where
\begin{align*}
\tau_k := \inf \{ t \geq 0 : \| Y_t^{(k)} \| > k \}, \quad k \in \mathbb{N}_0, 
\end{align*}
i.e. the solutions do not explode. Consequently, the process 
\begin{align*}
Y := y_0 \mathbbm{1}_{[\![ \tau_0 ]\!]} + \sum_{k \in \mathbb{N}} Y^{(k)} \mathbbm{1}_{]\!] \tau_{k-1}, \tau_k ]\!]}
\end{align*}
is a strong solution to (\ref{SDE-without-j}) with initial condition $y_0$.

Finally, for a general $\mathcal{F}_0$-measurable initial condition $y_0 : \Omega \rightarrow \mathcal{H}$, the process $Y := \sum_{k \in \mathbb{N}} Y^{(k)} \mathbbm{1}_{\Omega_k}$ is a strong solution to (\ref{SDE-without-j}) with initial condition $y_0$, where $(\Omega_k)_{k \in \mathbb{N}} \subset \mathcal{F}_0$ denotes the partition of $\Omega$ given by
$\Omega_k := \{ \| y_0 \| \in [k-1,k) \}$, and where for each $k \in \mathbb{N}$ the process $Y^{(k)}$ denotes a strong solution to (\ref{SDE-without-j}) with initial condition $y_0 \mathbbm{1}_{\Omega_k}$.
\end{proof}

\begin{theorem}\label{thm-SDE-existence}
We suppose that the mappings $(a,b,c|_B)$ are locally Lipschitz and satisfy the linear growth condition. Then, existence of strong solutions to (\ref{SDE}) holds.
\end{theorem}

\begin{proof}
Let $y_0 : \Omega \rightarrow \mathcal{H}$ be an arbitrary $\mathcal{F}_0$-measurable random variable. By induction, we will prove that for each $k \in \mathbb{N}_0$ there exists a local strong solution $Y^{(k)}$ to (\ref{SDE}) with initial condition $y_0$ and lifetime $\varrho_k$. By Lemma \ref{lemma-rho} we have $\varrho_0 = 0$, providing the assertion for $k = 0$. 

For the induction step $k \rightarrow k+1$ let $Y^{(k)}$ be a local strong solution to (\ref{SDE}) with initial condition $y_0$ and lifetime $\varrho_k$. Let $n \in \mathbb{N}$ be arbitrary and set $\Gamma_{kn} := \{ \varrho_k \in [n-1,n) \} \in \mathcal{F}_{\varrho_k}$. By Lemma \ref{lemma-Lipschitz-trans}, the mappings $(a^{(\varrho_k,\Gamma_{kn})},b^{(\varrho_k,\Gamma_{kn})},c^{(\varrho_k,\Gamma_{kn})}|_B)$ are locally Lipschitz, too. Therefore, by Proposition \ref{prop-SDE-existence} there exists a $\mathbb{F}^{(\varrho_k)}$-adapted strong a solution $Y^{(\varrho_k,\Gamma_{kn})-}$ to (\ref{SDE-without-j}) with parameters (\ref{para}), where $\Gamma = \Gamma_{kn}$, and initial condition $Y_{\varrho_k}^{(k)} \mathbbm{1}_{\Gamma_{kn}}$. By Lemma \ref{lemma-trans-solution-2}, the process $Y^{(\varrho_k,\Gamma_{kn})}$ defined according to (\ref{def-plus}) is a $\mathbb{F}^{(\varrho_k)}$-adapted local strong solution to (\ref{SDE}) with parameters (\ref{para}), where $\Gamma = \Gamma_{kn}$, initial condition $Y_{\varrho_k}^{(k)} \mathbbm{1}_{\Gamma_{kn}}$, and lifetime $\varrho_{k+1} - \varrho_k$. Noting that $(\Gamma_{kn})_{n \in \mathbb{N}}$ is a partition of $\Omega$, it follows that $Y^{(\varrho_k)} := \sum_{n \in \mathbb{N}} Y^{(\varrho_k,\Gamma_{kn})}$ is a $\mathbb{F}^{(\varrho_k)}$-adapted local strong solution to (\ref{SDE}) with initial condition $Y_{\varrho_k}^{(k)}$ and lifetime $\varrho_{k+1} - \varrho_k$. By Lemma \ref{lemma-solution-tau-2}, the process
\begin{align*}
Y^{(k+1)} := Y^{(k)} \mathbbm{1}_{[\![ 0,\varrho_k ]\!]} + Y_{\bullet - \varrho_k}^{(\varrho_k)} \mathbbm{1}_{]\!] \varrho_k,\varrho_{k+1} ]\!]}
\end{align*}
defined according to (\ref{def-Y-converse}) is a $\mathbb{F}$-adapted local strong solution to (\ref{SDE}) with initial condition $y_0$ and lifetime $\varrho_{k+1}$.

Consequently, for each $k \in \mathbb{N}_0$ there exists a local strong solution $Y^{(k)}$ to (\ref{SDE}) with initial condition $y_0$ and lifetime $\varrho_k$. By Lemma \ref{lemma-rho} we have $\mathbb{P}(\varrho_k \rightarrow \infty) = 1$. Hence, it follows that 
\begin{align*}
Y := y_0 \mathbbm{1}_{[\![ \varrho_0 ]\!]} + \sum_{k \in \mathbb{N}} Y^{(k)} \mathbbm{1}_{]\!] \varrho_{k-1}, \varrho_k ]\!]}
\end{align*}
is a $\mathbb{F}$-adapted strong solution to (\ref{SDE}) with initial condition $y_0$.
\end{proof}

\begin{theorem}\label{thm-SDE-ex-local}
We suppose that the mappings $(a,b,c|_B)$ are locally Lipschitz and locally bounded. Then, existence of local strong solutions to (\ref{SDE}) holds.
\end{theorem}

\begin{proof}
Let $y_0 : \Omega \rightarrow \mathcal{H}$ be an arbitrary $\mathcal{F}_0$-measurable random variable. We define the partition $(\Omega_k)_{k \in \mathbb{N}} \subset \mathcal{F}_0$ of $\Omega$ by $\Omega_k := \{ \| y_0 \| \in [k-1,k) \}$.
Furthermore, for each $k \in \mathbb{N}$ we define 
the mappings $a_k : \Omega \times \mathbb{R}_+ \times \mathcal{H} \rightarrow \mathcal{H}$, $b_k : \Omega \times \mathbb{R}_+ \times \mathcal{H} \rightarrow L_2^0(\mathcal{H})$ and $c_k : \Omega \times \mathbb{R}_+ \times \mathcal{H} \times E \rightarrow \mathcal{H}$ as in the proof of Proposition \ref{prop-SDE-existence}.
These mappings are locally Lipschitz and satisfy the linear growth condition. By Theorem \ref{thm-SDE-existence}, there exists a strong solution $Y^{(k)}$ to (\ref{SDE}) with parameters $a = a_k$, $b = b_k$ and $c = c_k$, and initial condition $y_0 \mathbbm{1}_{\Omega_k}$. The stopping time
\begin{align*}
\tau_k := \inf \{ t \geq 0 : \| Y_t^{(k)} \| > k \} 
\end{align*}
is strictly positive, and $Y^{(k)}$ is a local strong solution to (\ref{SDE}) with initial condition $y_0 \mathbbm{1}_{\Omega_k}$ and lifetime $\tau_k$. The stopping time $\tau := \sum_{k \in \mathbb{N}} \tau_k \mathbbm{1}_{\Omega_k}$ is strictly positive, and the process $Y := \sum_{k \in \mathbb{N}} Y^{(k)} \mathbbm{1}_{\Omega_k}$ is a local strong solution to (\ref{SDE}) with initial condition $y_0$ and lifetime $\tau$.
\end{proof}

\subsection{Comparison with the method of successive approximations}\label{sec-successive}

So far, our investigations provide the following result concerning existence and uniqueness of global strong solutions to the SDE (\ref{SDE}).

\begin{theorem}\label{thm-compare}
If $(a,b,c|_B)$ are locally Lipschitz and satisfy the linear growth condition, then existence and uniqueness of strong solutions to (\ref{SDE}) holds.
\end{theorem}

\begin{proof}
This is a direct consequence of Theorems \ref{thm-SDE-unique} and \ref{thm-SDE-existence}.
\end{proof}

Now, we shall provide a comparison with reference \cite{Cao}, where the authors also study Hilbert space valued SDEs of the type (\ref{SDE}). Their result \cite[Theorem 2.1]{Cao} is based on the method of successive approximations (see also \cite{Yamada, Taniguchi}) and considerably goes beyond the classical global Lipschitz conditions. For the sake of simplicity, let us recall the required assumptions in the time-homogeneous Markovian framework.  
In order to apply \cite[Theorem 2.1]{Cao}, for some constant $p \geq 2$ we need the estimate
\begin{equation}\label{cond-Zhang}
\begin{aligned}
&\| a(y_1) - a(y_2) \|^p + \| b(y_1) - b(y_2) \|_{L_2^0(\mathcal{H})}^p + \int_B \| c(y_1,x) - c(y_2,y) \|^p F(dx)
\\ &\quad + \bigg( \int_B \| c(y_1,x) - c(y_2,x) \|^2 F(dx) \bigg)^{p/2} \leq \kappa(\|y_1 - y_2\|^p) \quad \text{for all $y_1,y_2 \in \mathcal{H}$,}
\end{aligned}
\end{equation}
where $\kappa : \mathbb{R}_+ \rightarrow \mathbb{R}_+$ denotes a continuous, nondecreasing function with $\kappa(0) = 0$, and further conditions, which are precisely stated in \cite{Cao}, must be fulfilled. These conditions are satisfied if $\kappa$ is a continuous, nondecreasing and concave function such that
\begin{align}\label{int-at-zero-infty}
\int_0^{\epsilon} \frac{1}{\kappa(u)} du = \infty \quad \text{for each $\epsilon > 0$.}
\end{align}
In particular, we may choose $\kappa(u) = u$ for $u \in \mathbb{R}_+$, and consequently, both results, Theorem \ref{thm-compare} and \cite[Theorem 2.1]{Cao}, cover the classical situation, where global Lipschitz conditions are imposed.

However, there are situations where \cite[Theorem 2.1]{Cao} can be applied, while Theorem \ref{thm-compare} does not apply, and vice versa. For the sake of simplicity, in the following two examples we assume that $\mathcal{H} = \mathbb{R}$ and $b \equiv c \equiv 0$.

\begin{example}
We fix an arbitrary constant $0 < \delta < \exp(-1)$ and define the functions $\kappa,\rho : \mathbb{R}_+ \rightarrow \mathbb{R}_+$ by
\begin{align*}
\kappa(u) :=
\begin{cases}
0, & u=0,
\\ -u \ln u, & 0 < u < \delta,
\\ -\delta \ln \delta - (1 + \ln \delta) (u-\delta) & u \geq \delta,
\end{cases}
\end{align*}
as well as
\begin{align*}
\rho(u) :=
\begin{cases}
0, & u=0,
\\ u \sqrt{-\ln (u^2)}, & 0 < u < \sqrt{\delta},
\\ \sqrt{-\delta \ln \delta - (1 + \ln \delta) (u^2-\delta)} & u \geq \sqrt{\delta},
\end{cases}
\end{align*}
cf. \cite[Remark 1]{Yamada}. Let $a : \mathbb{R} \rightarrow \mathbb{R}$ be a mapping such that
\begin{align*}
|a(y_1) - a(y_2)| \leq \rho(|y_1 - y_2|) \quad \text{for all $y_1,y_2 \in \mathbb{R}$.}
\end{align*}
Then we have the estimate
\begin{align*}
|a(y_1) - a(y_2)|^2 \leq \kappa(|y_1 - y_2|^2) \quad \text{for all $y_1,y_2 \in \mathbb{R}$,}
\end{align*}
showing that condition (\ref{cond-Zhang}) with $p=2$ is satisfied. Moreover, $\kappa$ is a continuous, nondecreasing, concave function and condition (\ref{int-at-zero-infty}) is satisfied, because for each $0 < \epsilon < \delta$ we have
\begin{align*}
\int_0^{\epsilon} \frac{1}{\kappa(u)} du = -\int_0^{\epsilon} \frac{1}{u \ln u} du = -\ln |\ln u| \, \Big|_{u=0}^{u=\epsilon} = -\ln |\ln \epsilon| + \lim_{u \rightarrow 0} \ln |\ln u| = \infty.
\end{align*}
Consequently, \cite[Theorem 2.1]{Cao} applies. However, we have
\begin{align*}
\rho'(u) = \sqrt{-\ln(u^2)} - \frac{1}{\sqrt{-\ln(u^2)}} \quad \text{for $u \in (0,\sqrt{\delta})$,}
\end{align*}
and thus $\lim_{u \rightarrow 0} \rho'(u) = \infty$. Therefore, the mapping $a : \mathbb{R} \rightarrow \mathbb{R}$ might fail to be locally Lipschitz, and hence, Theorem \ref{thm-compare} does not apply.
\end{example}

\begin{example}
Let us define the mapping $a : \mathbb{R} \rightarrow \mathbb{R}$ as follows. For $n \in \mathbb{N}_0$ we define $a$ on the interval $[n,n+1]$ by
\begin{align*}
a(y) :=
\begin{cases}
n, & y \in [n, n+1 - \frac{1}{n+1}],
\\ n + (n+1) \big( y - (n+1 - \frac{1}{n+1}) \big), & y \in [n+1 - \frac{1}{n+1}, n+1].
\end{cases}
\end{align*}
This defines the mapping $a : \mathbb{R}_+ \rightarrow \mathbb{R}$, which we extend to a mapping $a : \mathbb{R} \rightarrow \mathbb{R}$ by symmetry
\begin{align*}
a(y) := a(-y), \quad y \in \mathbb{R}_-. 
\end{align*}
Then, $a$ is locally Lipschitz and satisfies the linear growth condition, and hence, Theorem \ref{thm-compare} applies. However, there are no constant $p \geq 2$ and no continuous, nondecreasing function $\kappa : \mathbb{R}_+ \rightarrow \mathbb{R}_+$ with $\kappa(0) = 0$ such that
\begin{align}\label{est-Zhang}
|a(y_1) - a(y_2)|^p \leq \kappa(|y_1 - y_2|^p) \quad \text{for all $y_1,y_2 \in \mathbb{R}$.}
\end{align}
Suppose, on the contrary, there exists a continuous, nondecreasing function  $\kappa : \mathbb{R}_+ \rightarrow \mathbb{R}_+$ with $\kappa(0) = 0$ fulfilling (\ref{est-Zhang}). Then we have
\begin{align}\label{kappa-geq-one}
|\kappa(u)| \geq 1 \quad \text{for all $u \in (0,1]$.}
\end{align}
Indeed, let $u \in (0,1]$ be arbitrary. Then, there exists $n \in \mathbb{N}$ with $\frac{1}{n} \leq u$. Moreover, by the definition of the mapping $a : \mathbb{R} \rightarrow \mathbb{R}$ there are $y_1,y_2 \in \mathbb{R}$ such that
\begin{align*}
|y_1 - y_2| \leq \bigg( \frac{1}{n} \bigg)^{1/p} \quad \text{and} \quad |a(y_1) - a(y_2)| = 1.
\end{align*}
Therefore, using the monotonicity of $\kappa$ and (\ref{est-Zhang}) we obtain
\begin{align*}
\kappa(u) \geq \kappa \bigg( \frac{1}{n} \bigg) \geq \kappa (|y_1 - y_2|^p) \geq |a(y_1) - a(y_2)|^p = 1,
\end{align*}
showing (\ref{kappa-geq-one}). Now, the continuity of $\kappa$ yields the contradiction $\kappa(0) \geq 1$. Consequently, condition (\ref{cond-Zhang}) is not satisfied, and thus, we cannot use \cite[Theorem~2.1]{Cao} in this case.
\end{example}

\section{Existence and uniqueness of mild solutions to Hilbert space valued SPDEs}\label{sec-SPDE}

In this section, we establish existence and uniqueness of (local) mild solutions to Hilbert space valued SPDEs of the type (\ref{SPDE}).

Let $H$ be a separable Hilbert space, let $(S_t)_{t \geq 0}$ be a $C_0$-semigroup on $H$ with infinitesimal generator $A : \mathcal{D}(A) \subset H \rightarrow H$, and let $B \in \mathcal{E}$ be a set with $F(B^c) < \infty$. Furthermore, let $\alpha : \Omega \times \mathbb{R}_+ \times H \rightarrow H$ and $\sigma : \Omega \times \mathbb{R}_+ \times H \rightarrow L_2^0(H)$ be $\mathcal{P} \otimes \mathcal{B}(H)$-measurable mappings, and let $\gamma : \Omega \times \mathbb{R}_+ \times H \times E \rightarrow H$ be a $\mathcal{P} \otimes \mathcal{B}(H) \otimes \mathcal{E}$-measurable mapping.

Throughout this section, we suppose that there exist another separable Hilbert space $\mathcal{H}$, a
$C_0$-group $(U_t)_{t \in \mathbb{R}}$ on $\mathcal{H}$ and
continuous linear operators $\ell \in L(H,\mathcal{H})$, $\pi \in
L(\mathcal{H},H)$ such that the diagram
\[ \begin{CD}
\mathcal{H} @>U_t>> \mathcal{H}\\
@AA\ell A @VV\pi V\\
H @>S_t>> H
\end{CD} \]
commutes for every $t \in \mathbb{R}_+$, that is
\begin{align}\label{diagram-commutes}
\pi U_t \ell = S_t \quad \text{for all $t \in \mathbb{R}_+$.}
\end{align}

\begin{remark}
According to \cite[Prop. 8.7]{SPDE}, this assumption is satisfied if the semigroup $(S_t)_{t \geq 0}$ is \emph{pseudo-contractive} (one also uses the notion \emph{quasi-contractive}), that is, there is a constant $\omega \in \mathbb{R}$ such that
\begin{align*}
\| S_t \| \leq e^{\omega t} \quad \text{for all $t \geq 0$.}
\end{align*}
This result relies on the Sz\H{o}kefalvi-Nagy theorem on unitary dilations (see e.g. \cite[Thm. I.8.1]{Nagy}, or \cite[Sec. 7.2]{Davies}). In the spirit of \cite{Nagy}, the group $(U_t)_{t \in \mathbb{R}}$ is called a \emph{dilation} of the semigroup $(S_t)_{t \geq 0}$.
\end{remark}

\begin{remark}
The Sz\H{o}kefalvi-Nagy theorem was also utilized in \cite{Seidler, Seidler2} in order to establish results concerning stochastic convolution integrals.
\end{remark}

Now, we define the mappings $a : \Omega \times \mathbb{R}_+ \times \mathcal{H} \rightarrow \mathcal{H}$, $b : \Omega \times \mathbb{R}_+ \times \mathcal{H} \rightarrow L_2^0(\mathcal{H})$ and $c : \Omega \times \mathbb{R}_+ \times \mathcal{H} \times E \rightarrow \mathcal{H}$ by
\begin{align}\label{def-a}
a(t,y) &:= U_{-t} \ell \alpha (t,\pi U_t y), 
\\ \label{def-b} b(t,y) &:= U_{-t} \ell \sigma (t,\pi U_t y),
\\ \label{def-c} c(t,y,x) &:= U_{-t} \ell \gamma (t,\pi U_t y,x).
\end{align}
Note that $a$ and $b$ are $\mathcal{P} \otimes \mathcal{B}(\mathcal{H})$-measurable, and that $c$ is $\mathcal{P} \otimes \mathcal{B}(\mathcal{H}) \otimes \mathcal{E}$-measurable.

\begin{lemma}\label{lemma-Lipschitz-transfer}
The following statements are true:
\begin{enumerate}
\item If $(\alpha,\sigma,\gamma|_B)$ are locally Lipschitz, then $(a,b,c|_B)$ are locally Lipschitz, too.

\item If $(\alpha,\sigma,\gamma|_B)$ satisfy the linear growth condition, then $(a,b,c|_B)$ satisfy the linear growth condition, too.

\item If $(\alpha,\sigma,\gamma|_B)$ are locally bounded, then $(a,b,c|_B)$ are locally bounded, too.
\end{enumerate}
\end{lemma}

\begin{proof}
All three statements are straightforward to check.
\end{proof}

\begin{proposition}\label{prop-SDE-SPDE}
Let $z_0 : \Omega \rightarrow H$ be a $\mathcal{F}_0$-measurable random variable, and let $\tau$ be a stopping time. Then, the following statements are true:
\begin{enumerate}
\item If $Y$ is a local strong solution to (\ref{SDE}) with initial condition $\ell z_0$ and lifetime $\tau$, then $Z := \pi U Y$ is a local mild solution to (\ref{SPDE}) with initial condition $z_0$ and lifetime $\tau$.

\item If $Z$ is a local mild solution to (\ref{SPDE}) with initial condition $z_0$ and lifetime $\tau$, then the process $Y$ defined as
\begin{equation}\label{def-Y-Z}
\begin{aligned}
Y_t &:= \ell z_0 + \int_0^{t \wedge \tau} U_{-s} \ell \alpha(s,Z_s) ds + \int_0^{t \wedge \tau} U_{-s} \ell \sigma(s,Z_s) dW_s
\\ &\quad + \int_0^{t \wedge \tau} \int_B U_{-s} \ell \gamma(s,Z_{s-},x) (\mu(ds,dx) - F(dx)ds)
\\ &\quad + \int_0^{t \wedge \tau} \int_{B^c} U_{-s} \ell \gamma(s,Z_{s-},x) \mu(ds,dx), \quad t \geq 0
\end{aligned}
\end{equation}
is a local strong solution to (\ref{SDE}) with initial condition $\ell z_0$ and lifetime $\tau$, and we have $Z^{\tau} = \pi U Y^{\tau}$.
\end{enumerate}
\end{proposition}

\begin{proof}
Let $Y$ be a local strong solution to (\ref{SDE}) with initial condition $\ell z_0$ and lifetime $\tau$. Then we have
\begin{align*}
&Z_{t \wedge \tau} = \pi U_{t \wedge \tau} Y_{t \wedge \tau} = \pi U_{t \wedge \tau} \bigg( \ell z_0 + \int_0^{t \wedge \tau} a(s,Y_s) ds + \int_0^{t \wedge \tau} b(s,Y_s) dW_s 
\\ &\quad + \int_0^{t \wedge \tau} \int_B c(s,Y_{s-},x) (\mu(ds,dx) - F(dx)ds) + \int_0^{t \wedge \tau} \int_{B^c} c(s,Y_{s-},x) \mu(ds,dx) \bigg).
\end{align*}
By the Definitions (\ref{def-a})--(\ref{def-c}) of $a,b,c$ we obtain
\begin{align*}
Z_{t \wedge \tau} &= \pi U_{t \wedge \tau} \bigg( \ell z_0 + \int_0^{t \wedge \tau} U_{-s} \ell \alpha(s,\pi U_s Y_s) ds + \int_0^{t \wedge \tau} U_{-s} \ell \sigma(s,\pi U_s Y_s) dW_s 
\\ &\quad + \int_0^{t \wedge \tau} \int_B U_{-s} \ell \gamma(s,\pi U_s Y_{s-},x) (\mu(ds,dx) - F(dx)ds) 
\\ &\quad + \int_0^{t \wedge \tau} \int_{B^c} U_{-s} \ell \gamma(s,\pi U_s Y_{s-},x) \mu(ds,dx) \bigg).
\end{align*}
Therefore, by (\ref{diagram-commutes}), and since $Z = \pi U Y$, we arrive at 
\begin{equation}\label{var-const-stopped}
\begin{aligned}
Z_{t \wedge \tau} &= S_{t \wedge \tau} z_0 + \int_0^{t \wedge \tau} S_{(t \wedge \tau)-s} \alpha(s,Z_s) ds + \int_0^{t \wedge \tau} S_{(t \wedge \tau)-s} \sigma(s,Z_s) dW_s 
\\ &\quad + \int_0^{t \wedge \tau} \int_B S_{(t \wedge \tau)-s} \gamma(s,Z_{s-},x) (\mu(ds,dx) - F(dx)ds) 
\\ &\quad + \int_0^{t \wedge \tau} \int_{B^c} S_{(t \wedge \tau)-s} \gamma(s,Z_{s-},x) \mu(ds,dx),
\end{aligned}
\end{equation}
showing that $Z$ is a local mild solution to (\ref{SPDE}) with initial condition $z_0$ and lifetime $\tau$. This establishes the first statement.
Now, let $Z$ be a local mild solution to (\ref{SPDE}) with initial condition $z_0$ and lifetime $\tau$. Then we have (\ref{var-const-stopped}), and therefore, by (\ref{diagram-commutes}) and the Definition (\ref{def-Y-Z}) of $Y$ we obtain
\begin{align*}
Z_{t \wedge \tau} &= \pi U_{t \wedge \tau} \bigg( \ell z_0 + \int_0^{t \wedge \tau} U_{-s} \ell \alpha(s,Z_s) ds + \int_0^{t \wedge \tau} U_{-s} \ell \sigma(s,Z_s) dW_s 
\\ &\quad + \int_0^{t \wedge \tau} \int_B U_{-s} \ell \gamma(s,Z_{s-},x) (\mu(ds,dx) - F(dx)ds) 
\\ &\quad + \int_0^{t \wedge \tau} \int_{B^c} U_{-s} \ell \gamma(s,Z_{s-},x) \mu(ds,dx) \bigg)
= \pi U_{t \wedge \tau} Y_{t \wedge \tau},
\end{align*}
showing that $Z^{\tau} = \pi U Y^{\tau}$. Therefore, by the Definition (\ref{def-Y-Z}) of $Y$ we obtain
\begin{align*}
Y_{t \wedge \tau} &= \ell z_0 + \int_0^{t \wedge \tau} U_{-s} \ell \alpha(s,\pi U_s Y_s) ds + \int_0^{t \wedge \tau} U_{-s} \ell \sigma(s,\pi U_s Y_s) dW_s
\\ &\quad + \int_0^{t \wedge \tau} \int_B U_{-s} \ell \gamma(s,\pi U_s Y_{s-},x) (\mu(ds,dx) - F(dx)ds)
\\ &\quad + \int_0^{t \wedge \tau} \int_{B^c} U_{-s} \ell \gamma(s,\pi U_s Y_{s-},x) \mu(ds,dx).
\end{align*}
Taking into account the Definitions (\ref{def-a})--(\ref{def-c}) of $a,b,c$, we get
\begin{align*}
&Y_{t \wedge \tau} = \ell z_0 + \int_0^{t \wedge \tau} a(s,Y_s) ds  + \int_0^{t \wedge \tau} b(s,Y_s) dW_s 
\\ &\quad + \int_0^{t \wedge \tau} \int_B c(s,Y_{s-},x) (\mu(ds,dx) - F(dx)ds) + \int_0^{t \wedge \tau} \int_{B^c} c(s,Y_{s-},x) \mu(ds,dx),
\end{align*}
showing that $Y$ is a local strong solution to (\ref{SDE}) with initial condition $\ell z_0$ and lifetime $\tau$.
\end{proof}

\begin{theorem}\label{thm-main}
The following statements are true:
\begin{enumerate}
\item If $(\alpha,\sigma,\gamma|_B)$ are locally Lipschitz and satisfy the linear growth condition, then existence and uniqueness of mild solutions to (\ref{SPDE}) holds.

\item If $(\alpha,\sigma,\gamma|_B)$ are locally Lipschitz and locally bounded, then existence and uniqueness of local mild solutions to (\ref{SPDE}) holds.

\item If $(\alpha,\sigma,\gamma|_B)$ are locally Lipschitz, then uniqueness of local mild solutions to (\ref{SPDE}) holds.
\end{enumerate}
\end{theorem}

\begin{proof}
Suppose that $(\alpha,\sigma,\gamma|_B)$ are locally Lipschitz. Let $Z$ and $Z'$ be two local mild solutions to (\ref{SPDE}) with initial conditions $z_0$ and $z_0'$, and lifetimes $\tau$ and $\tau'$. We define the $\mathcal{H}$-valued processes $Y$ and $Y'$ according to (\ref{def-Y-Z}). By Proposition \ref{prop-SDE-SPDE}, the processes $Y$ and $Y'$ are local strong solutions to (\ref{SDE}) with initial conditions $\ell z_0$ and $\ell z_0$, and lifetimes $\tau$ and $\tau'$, and we have $Z^{\tau} = \pi U Y^{\tau}$ and $(Z')^{\tau'} = \pi U (Y')^{\tau'}$. By Lemma \ref{lemma-Lipschitz-transfer}, the mappings $(a,b,c|_B)$ are also locally Lipschitz, and hence, Theorem \ref{thm-SDE-unique} yields that up to indistinguishability
\begin{align*}
Y^{\tau \wedge \tau'} \mathbbm{1}_{\{ \ell z_0 = \ell z_0' \}} = (Y')^{\tau \wedge \tau'} \mathbbm{1}_{\{ \ell z_0 = \ell z_0' \}}
\end{align*}
Therefore, we have up to indistinguishability
\begin{align*}
Z^{\tau \wedge \tau'} \mathbbm{1}_{\{ z_0 = z_0' \}} = \pi U Y^{\tau \wedge \tau'} \mathbbm{1}_{\{ z_0 = z_0' \}} = \pi U (Y')^{\tau \wedge \tau'} \mathbbm{1}_{\{ z_0 = z_0' \}} = (Z')^{\tau \wedge \tau'} \mathbbm{1}_{\{ z_0 = z_0' \}},
\end{align*}
proving uniqueness of local mild solutions to (\ref{SPDE}).

Now, we suppose that $(\alpha,\sigma,\gamma|_B)$ are locally Lipschitz and satisfy the linear growth condition. Let $z_0 : \Omega \rightarrow H$ be an arbitrary $\mathcal{F}_0$-measurable random variable. By Lemma \ref{lemma-Lipschitz-transfer}, the mappings $(a,b,c|_B)$ are also locally Lipschitz and satisfy the linear growth condition. Thus, 
by Theorem \ref{thm-SDE-existence} there exists a strong solution $Y$ to (\ref{SDE}) with initial condition $\ell z_0$. According to Proposition \ref{prop-SDE-SPDE}, the process $Z := \pi U Y$ is a mild solution to (\ref{SPDE}) with initial condition $z_0$, proving the existence of mild solutions to (\ref{SPDE}). 

If $(\alpha,\sigma,\gamma|_B)$ are locally Lipschitz and locally bounded, then a similar proof, which uses Theorem \ref{thm-SDE-ex-local}, shows that existence of local mild solutions to (\ref{SPDE}) holds.
\end{proof}

\begin{remark}
The structure $Z = \pi U Y$ shows that mild solutions to (\ref{SPDE}) obtained from Theorem \ref{thm-main} have c\`{a}dl\`{a}g sample paths.
\end{remark}

\begin{remark}
As pointed out in \cite{Marinelli-Prevot-Roeckner}, the existence of weak solutions to (\ref{SPDE}) relies on a suitable stochastic Fubini theorem. Sufficient conditions can be found in \cite{SPDE}.
\end{remark}

\end{document}